\documentstyle[11pt,amssymb,amsmath,amscd,amsbsy,amsfonts,theorem,hyperref,accents,bbm]{article}

\input xy
\xyoption{all}

\newcommand{\zz}{{\Bbb Z}}

\newcommand{\cc}{{\Bbb C}}

\newcommand{\pp}{{\Bbb P}}
\newcommand{\aaa}{{\Bbb A}}
\newcommand{\ff}{{\Bbb F}}

\newcommand{\ddeg}{\operatorname{deg}}
\newcommand{\ddet}{\operatorname{det}}

\newcommand{\Hom}{\operatorname{Hom}}

\newcommand{\op}[1]{\operatorname{#1}}

\newcommand{\row}{\rightarrow}

\newcommand{\lrow}{\longrightarrow}
\renewcommand{\leq}{\leqslant}
\renewcommand{\geq}{\geqslant}

\newcommand{\gm}{{\Bbb G}_m}

\newcommand{\calm}{{\cal M}}
\newcommand{\hm}{\operatorname{H}_{\calm}}
\newcommand{\nichego}[1]{}

\newcommand{\ov}[1]{\overline{#1}}

\newcommand{\wt}[1]{\widetilde{#1}}

\newcommand{\laz}{{\Bbb L}}

\newcommand{\cq}{{\cal Q}}

\newcommand{\sh}{{\cal SH}}

\newcommand{\shE}[1]{{\cal SH}_{{\Bbb{A}}^1}(#1)}
\newcommand{\shcE}[1]{{\cal SH}^c_{{\Bbb{A}}^1}(#1)}

\newcommand{\dmE}[1]{\op{DM}(#1)}

\newcommand{\moco}[2]{H_{{\cal M}}^{#1,#2}}

\newcommand{\Qed}{\hfill$\square$\smallskip}

\newenvironment{proof}{\noindent{\it Proof}:}{\vskip 5mm}

\newtheorem{prop}{Proposition}[section]{\bf}{\it}
\newtheorem{thm}[prop]{Theorem}{\bf}{\it}
\newtheorem{lem}[prop]{Lemma}{\bf}{\it}
{\bf}{\it}
{\bf}{\it}
{\bf}{\it}
{\bf}{\it}
{\bf}{\it}
\newtheorem{rem}[prop]{Remark}{\bf}{}
{\bf}{\it}

{\bf}{\it}
{\bf}{\it}
\newtheorem{cor}[prop]{Corollary}{\bf}{\it}
{\bf}{\it}

\begin{document}

\title{On torsion spaces
\renewcommand{\thefootnote}{\fnsymbol{footnote}} 
\renewcommand{\thefootnote}{\arabic{footnote}}
}
\author{Alexander Vishik}
\date{}
\maketitle

%\tableofcontents

\begin{abstract}
The purpose of this note is to construct examples of compact torsion objects of $\shE{F}$ of every $p$-level over an arbitrary field of characteristic different from $p$. We adapt
the approach of Mitchell \cite{Mit} to the algebraic situation.
We show that the respective level is determined by the action
of the $v_m$-elements on the $MGL$-motive, and also prove the refinement which permits to distinguish {\it isotropic Morava points}
of the Balmer spectrum of $\shcE{F}$.
\end{abstract}

\section{Introduction}

The comparison between $\sh$ and its' motivic category $D(Ab)$ can be visualised by the map of the respective Balmer spectra
(of the subcategories of compact objects) 
$\op{Spc}(D(Ab)^c)\row\op{Spc}(\sh^c)$, which is an embedding.
Here every (closed) point ${\frak a}_p\in\op{Spc}(D(Ab)^c)=\op{Spec}(\zz)$ is an end-point (and a specialisation) of an infinite ordered chain of primes  ${\frak a}_{p,n}$, $1\leq n\leq\infty$ of characteristic $p$ corresponding to Morava K-theories $K(p,n)$ (usually, denoted simply as $K(n)$) - see \cite[Corollary 9.5]{BalSSS}.
This computation is the fruit of the famous {\it Nilpotence Theorem} of Devinatz-Hopkins-Smith \cite{DHS}. The prime tensor-triangulated ideal ${\frak a}_{p,n}$ consists of compact
objects whose $n$-th Morava K-theory $K(n)$ is trivial (this automatically implies triviality of smaller Morava K-theories). Alternatively, these compact objects $Y$ are characterized by the property that $BP^*(Y\wedge Y^{\vee})$ is annihilated by some power of $v_n$, where the
latter is the standard generator of $BP=\zz_{(p)}[v_1,v_2,\ldots]$ of dimension $p^{n}-1$.  

The ideal ${\frak a}_{p,m}$ is embedded into ${\frak a}_{p,n}$, for $m>n$, and so, the former point is a specialization of the latter one. The fact that these points are indeed all different is established with the help of some ``test spaces'' $X_{p,n}$ whose
level is exactly $n$. In other words, $K(m)^*(X_{p,n})=0$, for $m<n$, and is non-zero, for $m\geq n$. There are various constructions of such spaces. One such elegant construction is
due to Mitchell \cite{Mit}. Here $X_{p,n}$ is realised as a direct summand of the suspension spectrum of a quotient of a general linear group $GL(\cc,p^n)$ by an $n$-dimensional $\ff_p$-vector space acting permutationally on the basis (with some modifications for $p=2$).

In this article we adapt the construction of Mitchell to the algebraic situation. This is pretty straightforward and most of the computations of Mitchell can be translated directly into algebraic geometry. But we stll need to use different tools of 
the algebro-geometric world in order to establish the needed properties of our objects. In particular, we had to adjust the
construction for $p=2$, because the orthogonal group (used in the
original construction) is not {\it special}, which makes the respective objects not handy (not $\zz/p$-cellular). 
Instead, we use the same construction as for odd $p$, but with
$n$ increased by one. We establish that (as in topology) the motive of $X_{p,n}$ is $\zz/p$-cellular and the respective motivic cohomology is a
free module over the Milnor subalgebra $\Lambda_H(Q_0,\ldots,Q_{n-1})$ generated by the first $n$ Milnor's operations.
This implies that the respective Morava K-theories $K(m)^{*,*'}$,
for $1\leq m<n$ are trivial on $X_{p,n}$ (since Milnor's operations $Q_m$ are exactly the first differentials in the respective Atiyah-Hirzebruch spectral sequences). Then, using
the mentioned spectral sequence and the information about the grading of our motivic cohomology we show that the Morava K-theories $K(m)^{*,*'}(X_{p,n})$ of our space
is non-trivial, for $m\geq n$. Here the double grading on motivic cohomology in algebraic geometry (as opposed to single grading in topology) comes very helpful and makes things much more transparent. This establishes that $X_{p,n}$ has the $p$-level exactly $n$ - Theorem \ref{main}. We also prove that the action of $v_m$ on the $MGL$-motive of $X_{p,n}$ is nilpotent, for $m<n$, and is not 
nilpotent, for $m\geq n$, - see Theorem \ref{MGL-Xpn}. 
Finally, we prove that the $p$-level of 
$X_{p,n}$ doesn't change, if we smash it with the (suspension spectrum of the) reduced \v{C}ech simplicial scheme $\wt{{\frak X}}_R$ of any $K(n)$-anisotropic variety $R$ - see Proposition \ref{Xpn-hi}. This latter result
permits to use our spaces $X_{p,n}$ to distinguish the
{\it level} of {\it isotropic Morava points} of 
$\op{Spc}(\shcE{F})$ - see
\cite{BsMV}.  

Over a field $\cc$, another example of a torsion space of $p$-type $n$ was given in \cite{Joach}. It implements the construction of Ravenel \cite[Appendix C]{Rav} in the algebro-geometric context.  

\medskip

\noindent
{\bf Acknowledgements:} I'm grateful to Burt Totaro for stimulating questions. The support of the EPSRC standard grant EP/T012625/1 is gratefully acknowledged.

\section{The construction of Mitchell}

We will construct a compact object $X_{p,n}$ of $\shE{F}_{(p)}$, such that the $m$-th Morava K-theory $K(m)$ is trivial on it, for any
$m<n$, over any extension $E/F$, while for $m\geq n$, it is non-trivial. 
In other words, $X_{p,n}\wedge K(m)$ is zero, for $m<n$, and non-zero, for $m\geq n$, which means that the $p$-level of our object is exactly $n$.
$X_{p,n}$ will be a direct summand in the suspension spectrum of some smooth variety over $F$.

\subsection{The case of an odd prime}
Let $V=(\zz/p)^{\times n}$ be an $n$-dimensional 
$\ff_p$-vector
space. Consider $S=\ff_p[y_1,\ldots,y_n]$ with every generator of degree $1$. We have a natural
action of $GL(V)$ on $S$. The classical result of Dickson then claims:

\begin{thm} {\rm(Dickson)}
 \label{Dickson}
 The subring of invariants $S^{GL(V)}$ is a polynomial
 algebra on generators $D_1,\ldots, D_n$, where 
 $\ddeg(D_r)=p^n-p^{n-r}$.
\end{thm}

Let $B,U$ and $\Sigma$ denote the Borel, unipotent and permutation subgroup respectively (corresponding to a given
basis of $V$). Following Steinberg \cite{St}, consider the
following idempotent $e_0\in\zz_{(p)}[GL(V)]$ in the group ring of $GL(V)$. 
$$
e_0=[GL(V):U]^{-1}\left(\sum_{b\in B}b\right)
\left(\sum_{\sigma\in\Sigma}sign(\sigma)\sigma\right).
$$
It defines an irreducible representation of $GL(V)$ of
dimension $p^{\binom{n}{2}}$. We have a natural group
homomorphism $\ddet:GL(V)\row\ff_p^{\times}$ which induces
a series of automorphisms $\phi_k$, $0\leq k\leq p-2$
of the ($p$-localised) group ring
$\zz_{(p)}[GL(V)]$ given by: $\phi_k(g)=\ddet^k(g^{-1})g$.
We obtain idempotents $e_k=\phi_k(e_0)$,
$0\leq k\leq p-2$. 

Let $p$ be odd and $F$ be some field of characteristic different from $p$. We have a natural embedding $V\row GL(p^n,F)$ given by the regular
representation of $V$ over $F$, which is part of a permutational representation of the group of affine
transformations of $V$ on $V$. 
The latter group is an extension
$$
1\row V\row Aff(V)\row GL(V)\row 1,
$$
and we have the natural action of $GL(V)$ on 
$V\backslash GL(p^n,F)$. 
Choose some $1\leq k\leq p-2$.
Following Mitchell \cite{Mit}, we define 
$Y:=\Sigma^{\infty}(V\backslash GL(p^n,F))\cdot e_k$ - the application of the projector $e_k$ to the suspension spectrum of $V\backslash GL(p^n,F)$.
We claim that $Y$ is a torsion space of $p$-level $n$, exactly as its topological counterpart.

Let $G$ be a linear algebraic group over $F$. We will denote as
$BG$ the {\it geometric} classifying space of $G$ over $F$
as defined in \cite{To} and \cite[Prop. 4.2.6]{MV} (see also \cite[Section 2.5]{SSWc}).
It classifies $G$-torsors in the etale topology. It comes
equipped with the principal $G$-bundle
$EG\row BG$, where both objects are colimits of smooth 
algebraic varieties and $EG$ is an open subscheme of an
infinite-dimensional affine space. 

Denote $GL(p^n,F)$ simply as $GL$. The embedding
$V\row GL$ leads to the cartesian
square in $\op{Spc}$:
\begin{equation}
 \label{car-sq}
 \xymatrix{
 \ov{V\backslash GL} \ar[r] \ar[d] & EGL \ar[d]\\
 BV \ar[r]^{\rho} & BGL
 },
\end{equation}
where $\rho$ is a morphism of ind-varieties and 
$\ov{V\backslash GL}=V\backslash GL\times EGL$. In particular,
the latter object can be identified in $\shE{F}$ with 
$V\backslash GL$.
The vertical maps here are principal $G$-bundles, so smooth
morphisms. 
In particular, in the category of motives over $BGL$ and $BV$
we have: $\rho^*M(EGL\row BGL)=M(\ov{V\backslash GL}\row BV)$.
Since $GL$ is a {\it special} group, the motive of the natural projection $M(BGL_{Nis}\row BGL)$ from the Nisnevich
classifying space to the {\it geometric} one coincides with
the trivial Tate-motive $M(id)$ - see 
\cite[Prop. 2.5.6]{SSWc} and \cite[Lemma 1.15]{MV}. By \cite[Prop. 3.2.7]{SSWc},
$$
M(EGL_{Nis}\row BGL_{Nis})=
\otimes_{i=1}^{p^n}\op{Cone}[-1](T\stackrel{c_i}{\row}T(i)[2i]),
$$
where $T$ is the trivial Tate-motive and $c_i$ is the $i$-th Chern class 
(an element of $\hm^{2i,i}(BGL,\zz)=\Hom_{DM(BGL)}(T,T(i)[2i])$). Hence, the same is true about the geometric variant $M(EGL\row BGL)$. Thus, we obtain:
\begin{equation*}
 M(\ov{V\backslash GL}\row BV)=\otimes_{i=1}^{p^n}
 \op{Cone}[-1](T\stackrel{\rho^*(c_i)}{\lrow}T(i)[2i]).
\end{equation*}
From \cite[Prop. 4.5]{Mit} we know that, with $\ff_p$-coefficients,
\begin{equation*}
\rho^*(c_i)=
\begin{cases}
&\pm D_r,\,\,\text{if}\,\,i=p^n-p^{n-r};\\
&0,\,\,\text{otherwise},
\end{cases}
\end{equation*}
where $\hm^{*,*'}(BV,\ff_p)=H[y_1,\ldots,y_n]\otimes
\Lambda(x_1,\ldots,x_n)$ with 
$H=\hm^{*,*'}(\op{Spec}(F),\ff_p)$, 
$\ddeg(y_i)=(1)[2]$,
$\ddeg(x_i)=(1)[1]$ - \cite[Theorem 6.10]{VoOP}, 
and $D_r$ are Dickson's invariants
in $y$-variables.
Thus, the mod $p$ motive is
$$
M(\ov{V\backslash GL}\row BV)=N\otimes\Lambda,
$$
where $\Lambda=\Lambda(z_i| 1\leq i\leq p^n, i\neq p^n-p^s)$,  $\ddeg(z_i)=(i)[2i-1]$ and
$$
N=\otimes_{r=1}^{n}\op{Cone}[-1](T\stackrel{D_r}{\row}
T(p^n-p^{n-r})[2(p^n-p^{n-r})]).
$$
Let $\pi:BV\row\op{Spec}(F)$ be the natural projection. Then 
$\pi_*(M(\ov{V\backslash GL}\row BV))=M(V\backslash GL)\in
\dmE{F}$. Thus, $M(V\backslash GL)$ is an external ``algebra'' over $\pi_*(N)$ in the prescribed generators.
We aim to compute the exact value of the $p$-level of $Y$, that is, to evaluate the triviality/non-triviality of various Morava K-theories of $Y$ over all extensions of $F$. By transfer arguments, we may assume that our field contains $p$-th roots of $1$, and so, $\zz/p$ has a $1$-dimensional faithful representation. Then,
considered with $\ff_p$-coefficients, the motive of $BV$
($=\pi_*(T)$) is $\zz/p[y_1^{\vee},\ldots,y_n^{\vee}]\otimes\Lambda(x_1^{\vee},\ldots,x_n^{\vee})$.
In particular, it is $\zz/p$-{\it cellular} (a direct sum of shifted copies of $\zz/p$).
Since $D_1,\ldots,D_n$ form a regular sequence in
$S=\ff_p[y_1,\ldots,y_n]$, from the exactness of the
Koszul complex it follows that (denoting $G:=GL(V)$)
$$
\pi_*(N)=\zz/p\otimes\big((S\otimes_{S^{G}}\ff_p)\otimes\Lambda(x_1,\ldots,x_n)\big)^{\vee}.
$$
So,
the motive of $V\backslash GL$ (with $\ff_p$-coefficients) is also $\zz/p$-cellular and its motivic cohomology is a free module over
$H$. 
If $char(F)=0$,
the topological realization functor maps 
$\zz/p(c)[d]$ to $\zz/p[d]$ and commutes with Milnor's operations $Q_i$ - see \cite{VoOP}. The projector $e_k$ defines a certain
bi-graded subspace in this bi-graded $\ff_p$-vector space.
From the result of Mitchell - \cite[Theorem 4.6]{Mit} we know
that the topological realization of it is a free module
over $\cq=\Lambda_{\ff_p}(Q_0,\ldots,Q_{n-1})$. This implies
that $\hm^{*,*'}((V\backslash GL)\cdot e_k,\ff_p)$ is a free module
over $\Lambda_H(Q_0,\ldots,Q_{n-1})$. 
If $char(F)$ is positive, we may consider the Chow-degree filtration $F_l:=\{\op{deg}_{Ch}\geq -l\}$ on $\Lambda$ (where $\op{deg}_{Ch}((i)[j])=j-2i$), which shows that as a $\cq$-module,
$\hm^{*,*'}((V\backslash GL)\cdot e_k,\ff_p)$ is an extension
of modules $H\otimes L=H\otimes\big((S\otimes_{S^{G}}\ff_p)\otimes\Lambda(x_1,\ldots,x_n)\big)\cdot e_k$ (cf. the case $p=2$ below), where the action comes from that on $\hm^{*,*'}(BV,\ff_p)$.
But for the latter motivic cohomology, the above assignment
$\zz(c)[d]\mapsto\zz[d]$ commutes with the action of (Steenrod and) Milnor's operations, since it is so for $B\zz/p$ by \cite{HKO}. And $L$ is a free module over $\cq$ by \cite[Theorem 3.1(b)]{Mit}. Thus, $\hm^{*,*'}((V\backslash GL)\cdot e_k,\ff_p)$ 
is, again, a free module over $\Lambda_H(Q_0,\ldots,Q_{n-1})$.
In particular,
Bockstein acts as an exact differential in this motivic cohomology, which means that the integral motive of
$(V\backslash GL)\cdot e_k$ is $p$-torsion, and hence,
$\zz/p$-cellular. The fact that our motivic cohomology
is a free module over $\cq$ means that the 
$p$-level of our object is $\geq n$, as the Milnor's operation
$Q_m$ is the first differential of the Atiyah-Hirzebruch spectral sequence converging to $K(m)^{*,*'}$.

To find the exact value of the level, observe that 
$((S\otimes_{S^{G}}\ff_p)\otimes\Lambda(x_1,\ldots,x_n))
\cdot e_k$ is
concentrated in bi-degrees $(i)[j]$ with $0\leq 2i-j\leq n$, and 
$Q_i$ decreases the number $2i-j$ by $1$. Since this module is free over $\cq$, by \cite[Theorem 3.1(b)]{Mit}, there is an element $\alpha$ of this module, which we will call a ``generator'' (since it is an element with the smallest $i$ and, simultaneously, the smallest $j$) of bidegree
$(kq)[2kq-n]$, where 
$\displaystyle q=\left(\frac{p^n-1}{p-1}\right)$. Then
$\beta:=Q_{n-1}\circ\ldots\circ Q_0(\alpha)$ is a non-zero element of bidegree $(i)[2i]$, for some $i$. Let $r\geq n$. Consider the Atiyah-Hirzebruch spectral sequence 
$$
\moco{*}{*'}(Y;\ff_p)\otimes_{\ff_p}\ff_p[v_r,v_r^{-1}]\Rightarrow K(r)^{*,*'}(Y),
$$
converging from motivic cohomology
of $Y$ to the $r$-th Morava K-theory of it.
By degree reason, all the differentials of this spectral sequence
vanish on $\beta$ (since the targets are above the line of slope two). On the other hand, since $\alpha$ was a ``generator'' and
the ``round size'' of these differentials is greater or equal than the size of $Q_r$, which in turn, is larger than that of 
$Q_{n-1}\circ\ldots\circ Q_0$, no differentials may hit $\beta$.
Thus, this element survives in $E_{\infty}$ and so,
$K(r)^{*,*'}(Y)\neq 0$. Hence, the $p$-level of $Y$ is exactly $n$. It remains to denote: $X_{p,n}:=Y$.

\begin{rem}
 The Mitchell's torsion space $Y$ is produced as a direct
 summand in the suspension spectrum of a smooth variety
 $V\backslash GL$. But this variety is not projective,
 so we don't get a {\it torsion motive} in the sense of
 \cite{TM}, as our object is not {\it pure}. Moreover, it
 is not particularly interesting from motivic point of
 view, since its motive is $\zz/p$-cellular (note, that a
 $\zz/p$-cellular motive can't be pure).
\end{rem}

\subsection{The case $p=2$}
In the case of $p=2$ we will deviate from the 
path chosen by Mitchell and will use instead the same
construction as for odd primes. The reason for this is that
our goals are somewhat different (we are interested only in the module structure over Milnor's subalgebra, but not in that
over the whole Steenrod algebra, and we are flexible with
the size of the object), while the original construction of
Mitchell meets obstacles due to the fact that the orthogonal
group (used there instead of $GL$) is not {\it special}.

Let $F$ be some field of characteristic different from $2$
and $V=(\zz/2)^{\times n}$, with $n\geq 2$. 
As above, we have a natural
embedding $V\row GL(2^n,F)$ given by the regular representation of $V$. Then there is an action of $GL(V)$
on $V\backslash GL(2^n,F)$. Let 
$Y=\Sigma^{\infty}(V\backslash GL(2^n,F))\cdot e_0$, where
$e_0$ is the Steinberg projector defined above.

Denote $GL(2^n,F)$ simply as $GL$.
In exactly the same way as for odd primes, we obtain that
\begin{equation*}
 M(\ov{V\backslash GL}\row BV)=\otimes_{i=1}^{2^n}
 \op{Cone}[-1](T\stackrel{\rho^*(c_i)}{\lrow}T(i)[2i]).
\end{equation*}
It still holds true that, with $\zz/2$-coefficients,
\begin{equation*}
\rho^*(c_i)=
\begin{cases}
&D_r,\,\,\text{if}\,\,i=2^n-2^{n-r};\\
&0,\,\,\text{otherwise},
\end{cases}
\end{equation*}
where $\hm^{*,*'}(BV,\zz/2)=H[y_1,\ldots,y_n,x_1,\ldots,x_n]/(x_i^2=y_i\tau+x_i\{-1\})$ - \cite[Theorem 6.10]{VoOP}, where $D_r$ are Dickson's invariants in the 
$y$-variables.
Indeed, for any $p$, odd or even, our embedding of groups
can be decomposed as $V\row{\Bbb{G}}_m^{\times n}\row GL$,
such that the pull-back of Chern roots of the tautological 
bundle on $GL$ will be exactly all elements of degree
$(1)[2]$ of the subalgebra $\ff_2[y_1,\ldots,y_n]$
(coming from the torus). The elementary symmetric functions of them will be exactly as described - see \cite[Lemma 1]{St-D}.
Thus, as for odd primes, our mod $2$ relative motive is
$$
M(\ov{V\backslash GL}\row BV)=N\otimes\Lambda,
$$
where $\Lambda=\Lambda(z_i| 1\leq i\leq 2^n, i\neq 2^n-2^s)$ with
$\ddeg(z_i)=(i)[2i-1]$ and
$$
N=\otimes_{r=1}^{n}\op{Cone}[-1](T\stackrel{D_r}{\row}
T(2^n-2^{n-r})[2^{n+1}-2^{n-r+1}]).
$$
The motive of $BV$ (with $\zz/2$-coefficients)
is still 
$\zz/2[y_1^{\vee},\ldots,y_n^{\vee}]\otimes\Lambda(x_1^{\vee},\ldots,x_n^{\vee})$
and the sequence $D_1,\ldots,D_n$ is still regular.
Hence, as above,
\begin{equation}
\label{piN}
\pi_*(N)=\zz/2\otimes\big((S\otimes_{S^{GL(V)}}\ff_2)\otimes\Lambda(x_1,\ldots,x_n)\big)^{\vee}.
\end{equation}
In particular, the mod $2$ motive of
$V\backslash GL$ is $\zz/2$-cellular.

If $char(F)=0$, denote as $\ov{y}_i$ and $\ov{x}_i$ the topological
realizations of classes $y_i$ and $x_i$.
Then $\ov{y}_i=\ov{x}_i^2$ and
$H^*(BV,\zz/2)=\ff_2[\ov{x}_1,\ldots,\ov{x}_n]=:S_0$.
Denote $GL(V)$ as $G$ and the Hopf algebra 
$\Lambda_{\ff_2}(Q_0,\ldots,Q_{n-2})$ of Milnor as $\cq$.
Denote as $\wt{D}_r$ the Dickson's invariants in the 
$\ov{x}_i$-variables. Then $\ov{(D_r)}=\wt{D}_r^2$ and
$\ov{S^{G}}=Fr_2(S_0^{G})$, where $Fr_2$ is the
$2$-Frobenius map. 
Note that $S_0$, $S_0^G$ and $Fr_2(S_0^G)$ are algebras
over the group Hopf algebra $\cq[G]$.
The topological realisation of
$(S\otimes_{S^{G}}\ff_2)\otimes\Lambda(x_1,\ldots,x_n)$
is $S_0\otimes_{Fr_2(S^{G}_0)}\ff_2$, which can be rewritten further as 
$S_0\otimes_{S_0^{G}}S_0^{G}
\otimes_{Fr_2(S_0^{G})}\ff_2$.
The module $S_0^{G}
\otimes_{Fr_2(S_0^{G})}\ff_2$ has a filtration by
$S_0^{G}$-submodules with trivial graded pieces.
Hence, as a module over $\cq[G]$, $S_0\otimes_{Fr_2(S^G_0)}\ff_2$ is an
extension of the shifted copies of 
$S_0\otimes_{S^{G}_0}\ff_2$. By the result of Mitchell
\cite[Theorem 3.1(c)]{Mit}, 
$(S_0\otimes_{S^{G}_0}\ff_2)\cdot e_0$
is a free module
over $\cq$. Hence, the
same is true about 
$(S_0\otimes_{Fr_2(S^{G}_0)}\ff_2)\cdot e_0$. 

Taking the Chow-degree filtration on $\Lambda$,
we obtain that as a $\cq[G]$-module,
$(S_0\otimes_{Fr_2(S^{G}_0)}\ff_2)\otimes\Lambda$ is an extension
of shifted copies of $S_0\otimes_{Fr_2(S^{G}_0)}\ff_2$.
Since $(S_0\otimes_{Fr_2(S^{G}_0)}\ff_2)\cdot e_0$
is a free module over $\cq$,
so is $((S_0\otimes_{Fr_2(S^{G}_0)}\ff_2)\otimes\Lambda)
\cdot e_0$. 

Finally, since $M(Y;\zz/2)$ is $\zz/2$-cellular and the 
topological realisation of it coincides with the latter
module, we obtain that $\hm^{*,*'}(Y,\ff_2)$ is a free
module over $\Lambda_H(Q_0,\ldots,Q_{n-2})$. 
If $char(F)$ is positive, considering the same Chow-degree filtration on $\Lambda$,
we still have that as a $\cq$-module, $\hm^{*,*'}(Y,\ff_2)$
is an extensions of the modules $H\otimes L$, with
$L=((S\otimes_{S^{G}}\ff_2)\otimes\Lambda(x_1,\ldots,x_n))\cdot e_0$, where
the action of $\cq$ comes from that on $\hm^{*,*'}(BV,\ff_2)$.
The latter action commutes with the ``pseudo-topological realisation'' $\zz(c)[d]\mapsto\zz[d]$ by \cite{HKO}. Then, as above we obtain that $\hm^{*,*'}(Y,\ff_2)$ is a free
module over $\Lambda_H(Q_0,\ldots,Q_{n-2})$.
In particular, $Q_0$ acts as an exact differential, which means that $Y$
is $2$-torsion, and so, $\zz_{(2)}$-motive of $Y$ is
$\zz/2$-cellular. The fact that $\hm^{*,*'}(Y,\zz/2)$
is a free module over $\cq$ implies that the $2$-level
of $Y$ is $\geq n-1$. 

To establish the exact value of the level, observe that
from the fact that $((S\otimes_{S^{G}}\ff_2)\otimes\Lambda(x_1,\ldots,x_n))\cdot e_0$ is concentrated in 
bi-degrees
$(i)[j]$ with $0\leq 2i-j\leq n$, the fact that this module
is free over $\cq$ and \cite[Theorem 3.1(c)]{Mit} it follows
that the ``generator'' $\alpha$ of this module (i.e., an element with the smallest $i$ and, simultaneously, the
smallest $j$) has bi-degree $(2^{n-1}-1)[2^n-1-n]$ (note that
each Milnor's operation $Q_s$ decreases the expression $2i-j$
by $1$). Then $\beta:=Q_{n-2}\circ\ldots\circ Q_0(\alpha)$
is non-zero and resides on the line of slope $=2$. 
Let $r\geq n-1$.
Since
$Y$ is a direct summand of the motive of a smooth variety,
all differentials $d_s$ of the Atiyah-Hirzebruch spectral sequence converging from $\hm^{*,*'}(Y,\zz/2)$ to the
$r$-th Morava K-theory $K(r)^{*,*'}(Y)$ vanish on $\beta$
by degree reasons (the targets are above the mentioned line). But, by the same reasons, no differentials
can hit $\beta$, since there is nothing below $\alpha$
(in the ``round'', or ``square'' sense) (note that the first
among these differentials is $d_1=Q_{r}$). Hence,
$K(r)^{*,*'}(Y)\neq 0$ and so, the level of $Y$ is exactly 
$(n-1)$. 
It remains to denote: $X_{2,n-1}:=Y$.

\section{The results}

Thus, we have proven (cf. \cite[Theorem B]{Mit}):

\begin{thm}
 \label{main}
 Let $p$ be prime and $F$ be a field of characteristic different from $p$. Then in $\shE{F}_{(p)}$ 
 there are compact
 $p$-torsion objects of an arbitrary $p$-level $n$. Such an
 object $X_{p,n}$ can be chosen as a direct summand of the suspension
 spectrum of a smooth variety $V\backslash GL(p^{\tilde{n}})$, with
 $V=(\zz/p)^{\tilde{n}}$, where $\tilde{n}=n$, for odd $p$, and $\tilde{n}=n+1$, for $p=2$.
\end{thm}

Let $M^{MGL}(X_{p,n})\in MGL_{(p)}\!-\!mod$ be the $MGL$-motive 
of our object. Recall that $MGL_{(p)}^{*,*'}$ is a polynomial algebra over the $BP^{*,*'}$-theory. Let $v_m\in BP$ be the standard generator of dimension $p^m-1$. Then we have:

\begin{thm}
 \label{MGL-Xpn}
 $v_m$ is nilpotent on $M^{MGL}(X_{p,n})$, for $m<n$, and is not
 nilpotent, for $m\geq n$.
\end{thm}

\begin{proof} 
 Suppose, $m<n$.
 Consider the Atiyah-Hirzebruch
 spectral sequence converging from motivic cohomology of $X_{p,n}$ to
 the $m$-th Morava K-theory of it. The first differential in it is the Milnor's operation $Q_m$. Since our motivic cohomology is a free module over $\Lambda_H(Q_m)$, this spectral sequence
 vanishes after the 1-st term, and so, 
 $K(m)^{*,*'}(X_{p,n})=0$. 
 
 Since we consider $MGL$-motives with $\zz_{(p)}$-coefficients,
 we may assume that $F$ contains a primitive $p$-th root $\zeta_p$ of $1$
 (by transfer arguments, since we can always acquire such a root
 by passing to an extension of degree a divisor of $(p-1)$).
 
 \begin{lem}
  \label{Xpn-Tate}
  Suppose, $\zeta_p\in F$. Then the motive $M^{MGL}(X_{p,n})\in MGL_{(p)}\!-\!mod$
  is an extension of finitely many Tate-motives.
 \end{lem}

 \begin{proof}
 Consider the $GL$-fibration $EGL\row BGL$ from the diagram
 (\ref{car-sq}). Denote: $d=p^{\tilde{n}}$. The chain of subgroups
 $$
 \{e\}\row GL_1\row GL_2\row\ldots\row GL_{d-1}\row GL_d
 $$
 gives the tower of fibrations
 $$
 EGL\row BGL_1\row BGL_2\row\ldots\row BGL_{d-1}\row BGL_d.
 $$
 The fibration $BGL_{r-1}\row BGL_r$ has the fiber
 $(\aaa^r\backslash 0)\times\aaa^{r-1}$ and so, $BGL_{r-1}$ is
 an $(r-1)$-dimensional bundle over a punctured (by the zero-section) $r$-dimensional bundle on $BGL_r$. Hence, the $MGL$-motive
 $M^{MGL}(BGL_{r-1}\row BGL_r)\in MGL(BGL_r)\!-\!mod$ is an extension of two Tate-motives. Consequently, the $MGL$-motive
 $M^{MGL}(EGL\row BGL)\in MGL(BGL)\!-\!mod$ is an extension of
 $2^d$ Tate-motives. A bit more careful analysis would reveal that this motive is, actually, isomorphic to
 $$
 \otimes_{i=1}^d\op{Cone}[-1](T\stackrel{c_i}{\lrow}T(i)[2i]),
 $$
 where $c_i\in MGL^{2i,i}(BGL)$ is the $i$-th Chern class, but
 we don't need such precision. From the cartesian square 
 (\ref{car-sq}), the motive of the fibration
 $\ov{V\backslash GL}\row BV$ is also an extension of $2^d$
 Tate-motives. 
 
 Since $\zeta_p\in F$,
 we can choose $EG\zz/p=\aaa^{\infty}\backslash 0$, which can be identified with the $\gm$-bundle $O(-1)$ on $\pp^{\infty}$. It has the natural action of $\zz/p$ (via scaling by $\zeta_p$). The quotient $B\zz/p$ is then naturally identified with the $\gm$-bundle $O(-p)$ on $\pp^{\infty}$. In other words, it
 is the complement to the zero section in the line bundle
 $O_{\pp^{\infty}}(-p)$. Hence,
 $$
 M^{MGL}(B\zz/p)=\op{Cone}[-1](M^{MGL}(\pp^{\infty})\stackrel{[-p](t)}{\lrow}M^{MGL}(\pp^{\infty})(1)[2]),
 $$
 where $t=c_1(O(1))$. In particular, it is an extension of Tate-motives, and the same holds for $M^{MGL}(BV)=M^{MGL}(B\zz/p)^{\otimes\tilde{n}}$. This implies that
 $M^{MGL}(V\backslash GL)=\pi_*(M^{MGL}(\ov{V\backslash GL}\row BV))$, where $\pi:BV\row\op{Spec}(F)$ is a natural projection, is also an extension of Tate-motives. On the other hand, $V\backslash GL$ is a finite-dimensional variety. Hence, its'
 motive is compact and so, $M^{MGL}(V\backslash GL)$ belongs
 to the thick triangulated subcategory ${\cal C}$ of $MGL(F)^c_{(p)}\!-\!mod$ generated by Tate-motives, and the same is true for 
 $M^{MGL}(X_{p,n})=M^{MGL}(V\backslash GL)\cdot e_k$. 
 \Qed
 \end{proof}
 
 On the subcategory ${\cal C}$ we have a non-degenerate weight structure in the sense of Bondarko \cite{Bon},
 whose heart ${\cal H}$ consists of (finite) direct sums of {\it pure} Tate-motives $T(i)[2i],\,i\in\zz$. 
 
 Now, we will prove that any object $U$ od ${\cal C}$ with
 zero Morava K-theory $K(m)^{*,*'}(U)$ is killed by an appropriate element of the form $v_m^s+\ldots$. We will do it by the induction on the length $k-l$ of the support interval $[l,k]$ of our object
 (that is, $U$ is an extention of objects from ${\cal H}[i]$,
 where $l\leq i\leq k$). If the length is negative, the object is zero and there is nothing to prove. We have an exact triangle
 $$
 U_{<k}\row U\row U_k\row U_{<k}[1],
 $$
 where $U_k\in {\cal H}[k]$ is the highest graded piece of $U$ and $U_{<k}$ is supported on $[l,k-1]$. Let $U_r=W_r[r]$, where
 $W_r\in{\cal H}$. Since $MGL^{*,*'}$ of a smooth variety vanishes above the line of slope two, we have:
 $MGL^{j,i}_{(p)}(U)=0$, for $j>2i+k$, while $MGL^{j,i}_{(p)}(U)=\op{Coker}(MGL^{j,i}_{(p)}(U_{k-1}[1])\row MGL^{j,i}_{(p)}(U_k))$, for $j=2i+k$.
 Also, $K(m)^{2*+k,*}(U)=MGL^{2*+k,*}_{(p)}(U)\otimes_{\laz_{(p)}}K(m)$, where $K(m)=\ff_p[v_m,v_m^{-1}]$ is the coefficient ring of 
 the pure part of our Morava K-theory. We know that $K(m)^{2*+k,*}(U)$ is zero. On the other hand, 
 $MGL^{2*+k,*}_{(p)}(U)$ is a finitely generated $\laz_{(p)}$-module. 
 There is the following fact:
 
 \begin{lem}
  \label{vm-etc}
  Let $N$ be a finitely generated $\laz_{(p)}$-module such that 
  $N\otimes_{\laz_{(p)}}K(m)=0$. Then $Ann_{\laz_{(p)}}(N)\otimes_{\laz_{(p)}}K(m)=K(m)$.
 \end{lem}

 \begin{proof}
  We will prove by induction on the number $n$ of generators of $N$ that $Ann_{\laz_{(p)}}(N)$ contains an element of the form $v_m^r+\ldots$, for some $r$. The base $n=1$ is obvious.
  Let $N=(y_1,\ldots,y_n)$, and 
  $I\subset\laz$ be the ideal generated by $p$ and all other generators, aside from $v_m$ (i.e., generators $x_i$ of dimension $i\neq p^m-1$). Since 
  $N\otimes_{\laz_{(p)}}K(m)=0$, $y_i\cdot v_m^{r_i}=\sum_{j=1}^ny_j\cdot u_{j,i}$, where $u_{j,i}\in I$. Then
  $y_i(v_m^{r_i}-u_{i,i})=\sum_{j\neq i}y_j\cdot u_{j,i}$ and 
  $y_i((v_m^{r_i}-u_{i,i})(v_m^{r_n}-u_{n,n})-u_{n,i}u_{i,n})=
  \sum_{j\neq i,n}y_j\cdot w_j$, for some $w_j\in I$. Hence,
  $N_{(1,\ldots,n-1)}=(y_1,\ldots,y_{n-1})$ also satisfies:
  $N_{(1,\ldots,n-1)}\otimes_{\laz}K(m)=0$. By inductive 
  assumption, $N_{(1,\ldots,n-1)}$ is annihilated by an element of the form $v_m^s+\ldots$. The same is true for $N_{(2,\ldots,n)}$, and so, for $N$. The induction step and the statement are proven.
 \Qed
 \end{proof}
 
 Recalling that $W_k$ is a sum of pure Tate-motives and considering
 $N= MGL^{2*+k,*}_{(p)}(U\otimes W_k^{\vee})$ it follows from Lemma \ref{vm-etc} that there exists $\tilde{v}=v_m^r+\ldots$, such that the map $\cdot\tilde{v}:W_k\row W_k(*)[2*]$ factors
 through the map $W_k\row W_{k-1}$. 
 Observe that $l<k$, since Tate-motives have non-trivial
 Morava K-theory. If $l=k-1$, then 
 $U=\op{Cone}(W_k\row W_{k-1})[k-1]$, and considering the dual
 object $U^{\vee}$, we obtain that there exists
 $\tilde{v}'=v_m^{r'}+\ldots$,
 such that the map $\cdot\tilde{v}':W_{k-1}\row W_{k-1}(*')[2*']$ factors through the map $W_k\row W_{k-1}$. Both conditions combined imply that the multiplication by $\tilde{v}'\cdot\tilde{v}$ is zero on $U$. So, we may assume
 that $l<k-1$. Let 
 $Y=\op{Cone}[-1](W_k\stackrel{\cdot\tilde{v}}{\lrow}W_k(*)[2*])$. Then we have a map of exact triangles:
 $$
 \xymatrix{
 U_{<k} \ar[d] \ar[r] & U \ar[d] \ar[r] & W_k[k] \ar@{=}[d] \ar[r] & U_{<k}[1] \ar[d] \\
 W_k(*)[2*+k-1] \ar[r] & Y[k] \ar[r] & W_k[k] \ar[r]^(0.4){\cdot\tilde{v}} & W_k(*)[2*+k]
 }
 $$
 Let $Z=\op{Cone}[-1](U\row Y[k])$. Then $Z$ has support $[l,k-1]$ (since $l<k-1$). As $\tilde{v}$ is invertible in $K(m)$, the
 Morava K-theory $K(m)^{*,*'}(Y)$ is zero, and hence, 
 $K(m)^{*,*'}(Z)=0$. Note that $Y=W_k\otimes\op{Cone}[-1](T\stackrel{\cdot\tilde{v}}{\row}T(*)[2*])$ is killed by $\tilde{v}^2$ by \cite[Proposition 2.13]{BalSSS}. Hence,
 if $Z$ is killed by $w\in\laz$, then $U$ is killed by
 $\tilde{v}^2\cdot w$. By induction on the support, it follows
 that any object $U$ of ${\cal C}$ with trivial $K(m)^{*,*'}$ is 
 killed by an element of the form $v_m^r+\ldots$. 
 
 Since our object $M^{MGL}(X_{p,n})$ comes from $\shE{F}_{(p)}$,
 the annihilator of it is an ideal of $\laz$ invariant under 
 Landweber-Novikov operations. By the above,
 this ideal contains some element $v_m^r+...$. By the
 result of Landweber \cite[Proposition 2.11]{La73b}, it contains
 then some power of $v_m$.
 
 The fact that, for $m\geq n$, $v_m$ is not nilpotent on
 $M^{MGL}(X_{p,n})$ is clear, since $K(m)^{*,*'}(X_{p,n})\neq 0$,
 in this case. This finishes the proof of Theorem \ref{MGL-Xpn}.
 \Qed
\end{proof}

We also have the following refinement of Theorem \ref{main}, which is used in the study
of the Balmer spectrum of Morel-Voevodsky category - \cite{BsMV}.

Let $R/F$ be some smooth variety and ${\frak X}_R$ be the suspension spectrum of the \v{C}ech simplicial scheme of $R$. 
This is a $\wedge$-projector, while the natural map to the point 
gives the distinguished triangle
$$
{\frak X}_R\row{\mathbbm{1}}\row\wt{{\frak X}}_R\row
{\frak X}_R[1]
$$
in $\shE{F}$, where $\wt{{\frak X}}_R$ is the complementary projector. 

Denote as $K(m)^*:=K(m)^{2*,*}$ the {\it pure part} of our Morava
K-theory considered as an oriented cohomology theory on the
category ${\mathbf Sm}_F$ of smooth varieties over $F$.
Recall from \cite[Definition 2.1]{INCHKm}, that a smooth projective variety $R\stackrel{\pi}{\row}\op{Spec}(F)$ is called {\it $K(m)$-anisotropic}, if
the push-forward map $\pi_*:K(m)_*(R)\row K(m)_*(\op{Spec}(F))$ is zero.

Let $X_{p,n}^{\vee}$ be the dual to our compact
object $X_{p,n}$.

\begin{prop}
 \label{Xpn-hi}
 For any smooth projective $K(r)$-anisotropic variety $R/F$, the map
 $$
 K(r)_{2*,*}(X_{p,n}^{\vee})\row
 K(r)_{2*,*}(X_{p,n}^{\vee}\wedge \wt{{\frak X}}_R)
 $$ 
 is
 an isomorphism. In particular, for any $r\geq n$, the pure part of the 
 Morava K-theory $K(r)_{2*,*}(X_{p,n}^{\vee}\wedge \wt{{\frak X}}_R)$ is non-zero. 
\end{prop}

\begin{proof}
By transfer arguments we may assume that $F$ contains a primitive
$p$-th root $\zeta_p$ of $1$. Then, by Lemma \ref{Xpn-Tate},
the $MGL$-motive $M^{MGL}(X_{p,n})$ of our torsion space is an
extension of finitely many Tate-motives. On the Tate-motivic subcategory of $MGL(F)_{(p)}^c-mod$ we have a weight structure
in the sense of Bondarko, with the heart ${\cal H}$ consisting of direct sums of pure Tate motives. Since $X_{p,n}$ is a direct summand 
in (the suspension spectrum of) a smooth variety, its' $MGL$-motive belongs
to $D_{\leq 0}$, and so $M^{MGL}(X_{p,n}^{\vee})\in D_{\geq 0}$.
Let $\ldots\row W_1\stackrel{f}{\row}W_0$ be the respective weight complex (corresponding to some choice of the weight filtration), where $W_i$ is a direct sum of pure Tate-motives.
Since Morava K-theory $K(r)^{*,*'}$ of a smooth variety is zero
above the line of slope $2$, we have:
$$
K(r)_{2*,*}(X_{p,n}^{\vee})=\op{Coker}(K(r)_{2*,*}(W_1)
\stackrel{f_*}{\row}K(r)_{2*,*}(W_0)).
$$
Similarly, since $M^{MGL}({\wt{\frak X}}_R)$ 
is an extension of $\op{Cone}(M^{MGL}(R)\row{\mathbbm{1}}^{MGL})$ and the rest which
belongs to the $D_{>1}$ part of the weight filtration on the $MGL$-modules, we have:
$$
K(r)_{2*,*}(X_{p,n}^{\vee}\wedge \wt{{\frak X}}_R)=
\op{Coker}(K(r)_{2*,*}(W_1)\oplus K(r)_{2*,*}(W_0\otimes R)
\stackrel{(f_*,\pi_*)}{\row}K(r)_{2*,*}(W_0)).
$$
But since $W_0$ is a direct sum of Tate-motives and $R$ is
$K(r)$-anisotropic, the map $\pi_*$ here is zero. Thus,
$K(r)_{2*,*}(X_{p,n}^{\vee}\wedge \wt{{\frak X}}_R)=
K(r)_{2*,*}(X_{p,n}^{\vee})=K(r)^{-2*,-*}(X_{p,n})$ and the latter group is non-zero
as was shown in the proof of Theorem \ref{main} (the constructed
non-trivial element $\beta$ was pure).
 \Qed
\end{proof}

This immediately implies:

\begin{cor}
 The action of $v_m$ is not nilpotent on the $MGL$-motive of 
 $X_{p,n}^{\vee}\wedge \wt{{\frak X}}_R$, for any
 $m\geq n$ and any $K(m)$-anisotropic variety $R$.
\end{cor}

\bigskip

\begin{itemize}
\item[address:] {\small School of Mathematical Sciences, University of Nottingham, University Park, Nottingham, NG7 2RD, UK}
\item[email:] {\small\ttfamily alexander.vishik@nottingham.ac.uk}
\end{itemize}

\end{document}